\documentclass[10pt]{amsart}
\usepackage{amsmath,amssymb,latexsym,hyperref,url,graphicx}

\newtheorem*{conjecture}{Conjecture}
\newtheorem*{lemma}{Lemma}
\newtheorem{proposition}{Proposition}
\newtheorem*{theorem}{Theorem}

\newcommand{\Z}{\mathbb{Z}}
\newcommand{\Mod}{\textnormal{mod}}
\newcommand{\lcm}{\operatorname*{lcm}}

\newcommand{\seqnum}[1]{\href{http://www.research.att.com/~njas/sequences/#1}{\underline{#1}}}
\newcommand{\page}[1]{page~#1}
\newcommand{\eqn}[1]{Equation~(\ref{#1})}
\newcommand{\scn}[1]{Section~\ref{#1}}

\begin{document}

\title{A natural prime-generating recurrence}
\author{Eric S. Rowland\\
	Department of Mathematics\\
	Rutgers University\\
	Piscataway, NJ 08854, USA
}
\date{July 20, 2008}

\begin{abstract}
For the sequence defined by $a(n) = a(n-1) + \gcd(n,a(n-1))$ with $a(1) = 7$ we prove that $a(n) - a(n-1)$ takes on only $1$s and primes, making this recurrence a rare ``naturally occurring'' generator of primes.  Toward a generalization of this result to an arbitrary initial condition, we also study the limiting behavior of $a(n)/n$ and a transience property of the evolution.
\end{abstract}

\maketitle
\markboth{Eric Rowland}{A natural prime-generating recurrence}

\section{Introduction}\label{introduction}

Since antiquity it has been intuited that the distribution of primes among the natural numbers is in many ways random.  For this reason, functions that reliably generate primes have been revered for their apparent traction on the set of primes.

Ribenboim \cite[\page{179}]{ribenboim} provides three classes into which certain prime-generating functions fall:
\begin{enumerate}
	\item[(a)] $f(n)$ is the $n$th prime $p_n$.
	\item[(b)] $f(n)$ is always prime, and $f(n) \neq f(m)$ for $n \neq m$.
	\item[(c)] The set of positive values of $f$ is equal to the set of prime numbers.
\end{enumerate}
Known functions in these classes are generally infeasible to compute in practice.  For example, both Gandhi's formula
\[
	p_n = \left\lfloor 1 - \log_2 \left( -\frac{1}{2} + \sum_{d \mid P_{n-1}} \frac{\mu(d)}{2^d-1} \right) \right\rfloor
\]
\cite{gandhi}, where $P_n = p_1 p_2 \cdots p_n$, and Willans' formula
\[
	p_n = 1 + \sum_{i=1}^{2^n} \left\lfloor \left(\frac{n}{\sum_{j=1}^i \left\lfloor\left(\cos \frac{(j-1)! + 1}{x} \pi\right)^2\right\rfloor }\right)^{1/n} \right\rfloor
\]
\cite{willans} satisfy condition (a) but are essentially versions of the sieve of Eratosthenes \cite{golomb74,gw}.  Gandhi's formula depends on properties of the M\"obius function $\mu(d)$, while Willans' formula is built on Wilson's theorem.  Jones \cite{jones} provided another formula for $p_n$ using Wilson's theorem.

Functions satisfying (b) are interesting from a theoretical point of view, although all known members of this class are not practical generators of primes.  The first example was provided by Mills \cite{mills}, who proved the existence of a real number $A$ such that $\lfloor A^{3^n} \rfloor$ is prime for $n \geq 1$.  The only known way of finding an approximation to a suitable $A$ is by working backward from known large primes.  Several relatives of Mills' function can be constructed similarly \cite{dudley}.

The peculiar condition (c) is tailored to a class of multivariate polynomials constructed by Matiyasevich \cite{matiyasevich} and Jones et al.\ \cite{jsww} with this property.  These results are implementations of primality tests in the language of polynomials and thus also cannot be used to generate primes in practice.

It is evidently quite rare for a prime-generating function to not have been expressly \emph{engineered} for this purpose.  One might wonder whether there exists a nontrivial prime-generating function that is ``naturally occurring'' in the sense that it was not constructed to generate primes but simply \emph{discovered} to do so.

Euler's polynomial $n^2 - n + 41$ of 1772 is presumably an example; it is prime for $1 \leq n \leq 40$.  Of course, in general there is no known simple characterization of those $n$ for which $n^2 - n + 41$ is prime.  So, let us revise the question:  Is there a naturally occurring function that always generates primes?

The subject of this paper is such a function.  It is recursively defined and produces a prime at each step, although the primes are not distinct as required by condition (b).

The recurrence was discovered in 2003 at the NKS Summer School\footnote{
The NKS Summer School (\url{http://www.wolframscience.com/summerschool}) is a three-week program in which participants conduct original research informed by \emph{A New Kind of Science} \cite{nks}.
}, at which I was a participant.  Primary interest at the Summer School is in systems with simple definitions that exhibit complex behavior.  In a live computer experiment led by Stephen Wolfram, we searched for complex behavior in a class of nested recurrence equations.  A group led by Matt Frank followed up with additional experiments, somewhat simplifying the structure of the equations and introducing different components.  One of the recurrences they considered is
\begin{equation}\label{rec}
	a(n) = a(n-1) + \gcd(n, a(n-1)).
\end{equation}
They observed that with the initial condition $a(1) = 7$, for example, the sequence of differences $a(n) - a(n-1) = \gcd(n, a(n-1))$ (sequence \seqnum{A132199}) appears chaotic \cite{frank}.  When they presented this result, it was realized that, additionally, this difference sequence seems to be composed entirely of $1$s and primes:

\ % extra space before data

\footnotesize
\noindent 1, 1, 1, 5, 3, 1, 1, 1, 1, 11, 3, 1, 1, 1, 1, 1, 1, 1, 1, 1, 1, 23,
3, 1, 1, 1, 1, 1, 1, 1, 1, 1, 1, 1, 1, 1, 1, 1, 1, 1, 1, 1, 1, 1, 1,
47, 3, 1, 5, 3, 1, 1, 1, 1, 1, 1, 1, 1, 1, 1, 1, 1, 1, 1, 1, 1, 1, 1,
1, 1, 1, 1, 1, 1, 1, 1, 1, 1, 1, 1, 1, 1, 1, 1, 1, 1, 1, 1, 1, 1, 1,
1, 1, 1, 1, 1, 1, 1, 1, 101, 3, 1, 1, 7, 1, 1, 1, 1, 11, 3, 1, 1, 1,
1, 1, 13, 1, 1, 1, 1, 1, 1, 1, 1, 1, 1, 1, 1, 1, 1, 1, 1, 1, 1, 1, 1,
1, 1, 1, 1, 1, 1, 1, 1, 1, 1, 1, 1, 1, 1, 1, 1, 1, 1, 1, 1, 1, 1, 1,
1, 1, 1, 1, 1, 1, 1, 1, 1, 1, 1, 1, 1, 1, 1, 1, 1, 1, 1, 1, 1, 1, 1,
1, 1, 1, 1, 1, 1, 1, 1, 1, 1, 1, 1, 1, 1, 1, 1, 1, 1, 1, 1, 1, 1, 1,
1, 1, 1, 1, 1, 1, 1, 1, 1, 1, 1, 1, 1, 1, 1, 1, 1, 1, 1, 1, 1, 1, 1,
1, 1, 1, 233, 3, 1, 1, 1, 1, 1, 1, 1, 1, 1, 1, 1, 1, 1, 1, 1, 1, 1,
1, 1, 1, 1, 1, 1, 1, 1, 1, 1, 1, 1, 1, 1, 1, 1, 1, 1, 1, 1, 1, 1, 1,
1, 1, 1, 1, 1, 1, 1, 1, 1, 1, 1, 1, 1, 1, 1, 1, 1, 1, 1, 1, 1, 1, 1,
1, 1, 1, 1, 1, 1, 1, 1, 1, 1, 1, 1, 1, 1, 1, 1, 1, 1, 1, 1, 1, 1, 1,
1, 1, 1, 1, 1, 1, 1, 1, 1, 1, 1, 1, 1, 1, 1, 1, 1, 1, 1, 1, 1, 1, 1,
1, 1, 1, 1, 1, 1, 1, 1, 1, 1, 1, 1, 1, 1, 1, 1, 1, 1, 1, 1, 1, 1, 1,
1, 1, 1, 1, 1, 1, 1, 1, 1, 1, 1, 1, 1, 1, 1, 1, 1, 1, 1, 1, 1, 1, 1,
1, 1, 1, 1, 1, 1, 1, 1, 1, 1, 1, 1, 1, 1, 1, 1, 1, 1, 1, 1, 1, 1, 1,
1, 1, 1, 1, 1, 1, 1, 1, 1, 1, 1, 1, 1, 1, 1, 1, 1, 1, 1, 1, 1, 1, 1,
1, 1, 1, 1, 1, 1, 1, 1, 1, 1, 1, 1, 1, 1, 1, 1, 1, 1, 1, 1, 1, 1, 1,
1, 1, 1, 1, 1, 1, 1, 1, 467, 3, 1, 5, 3, 1, 1, 1, 1, 1, 1, 1, 1, 1,
1, 1, 1, 1, 1, 1, 1, 1, 1, 1, 1, 1, 1, 1, 1, 1, \dots
\normalsize

\ % extra space after data

While the recurrence certainly has something to do with factorization (due to the $\gcd$), it was not clear why $a(n) - a(n-1)$ should never be composite.  The conjecture was recorded for the initial condition $a(1) = 8$ in sequence \seqnum{A084663}.

The main result of the current paper is that, for small initial conditions, $a(n) - a(n-1)$ is always $1$ or prime.  The proof is elementary; our most useful tool is the fact that $\gcd(n, m)$ divides the linear combination $r n + s m$ for all integers $r$ and $s$.

At this point the reader may object that the $1$s produced by $a(n) - a(n-1)$ contradict the previous claim that the recurrence \emph{always} generates primes.  However, there is some local structure to $a(n)$, given by the lemma in \scn{local}, and the length of a sequence of $1$s can be determined at the outset.  This provides a shortcut to simply skip over this part of the evolution directly to the next nontrivial $\gcd$.  By doing this, one produces the following sequence of primes (sequence \seqnum{A137613}).

\ % extra space before data

\footnotesize
\noindent 5, 3, 11, 3, 23, 3, 47, 3, 5, 3, 101, 3, 7, 11, 3, 13, 233, 3, 467,
3, 5, 3, 941, 3, 7, 1889, 3, 3779, 3, 7559, 3, 13, 15131, 3, 53, 3,
7, 30323, 3, 60647, 3, 5, 3, 101, 3, 121403, 3, 242807, 3, 5, 3, 19,
7, 5, 3, 47, 3, 37, 5, 3, 17, 3, 199, 53, 3, 29, 3, 486041, 3, 7,
421, 23, 3, 972533, 3, 577, 7, 1945649, 3, 163, 7, 3891467, 3, 5, 3,
127, 443, 3, 31, 7783541, 3, 7, 15567089, 3, 19, 29, 3, 5323, 7, 5,
3, 31139561, 3, 41, 3, 5, 3, 62279171, 3, 7, 83, 3, 19, 29, 3, 1103,
3, 5, 3, 13, 7, 124559609, 3, 107, 3, 911, 3, 249120239, 3, 11, 3, 7,
61, 37, 179, 3, 31, 19051, 7, 3793, 23, 3, 5, 3, 6257, 3, 43, 11, 3,
13, 5, 3, 739, 37, 5, 3, 498270791, 3, 19, 11, 3, 41, 3, 5, 3,
996541661, 3, 7, 37, 5, 3, 67, 1993083437, 3, 5, 3, 83, 3, 5, 3, 73,
157, 7, 5, 3, 13, 3986167223, 3, 7, 73, 5, 3, 7, 37, 7, 11, 3, 13, 17, 3, \dots
\normalsize

\ % extra space after data

It certainly seems to be the case that larger and larger primes appear fairly frequently.  Unfortunately, these primes do not come for free:  If we compute terms of the sequence without the aforementioned shortcut, then a prime $p$ appears only after $\frac{p-3}{2}$ consecutive $1$s, and indeed the primality of $p$ is being established essentially by trial division.  As we will see, the shortcut is much better, but it requires an external primality test, and in general it requires finding the smallest prime divisor of an integer $\Delta$.  So although it is naturally occurring, the recurrence, like its artificial counterparts, is not a magical generator of large primes.

We mention that Benoit Cloitre \cite{cloitre} has considered variants of \eqn{rec} and has discovered several interesting results.  A striking parallel to the main result of this paper is that if
\[
	b(n) = b(n-1) + \lcm(n, b(n-1))
\]
with $b(1) = 1$, then $b(n)/b(n-1)-1$ (sequence \seqnum{A135506}) is either $1$ or prime for each $n \geq 2$.

%The theorem of the present paper bears resemblance to a conjecture of Fortune, which states that the distance between $P_n$ and the smallest prime greater than $P_n$ is either $1$ or prime \cite{golomb81}.  This conjecture is still open.

\begin{table}
\footnotesize
\begin{center}
\begin{tabular}{c @{\qquad\qquad} c}
$\begin{array}{rrcrl}
n & \Delta(n) & g(n) & a(n) & a(n)/n \\ \hline
1 &  &  & 7 & 7 \\
2 & 5 & 1 & 8 & 4 \\
3 & 5 & 1 & 9 & 3 \\
4 & 5 & 1 & 10 & 2.5 \\
5 & 5 & 5 & 15 & 3 \\
6 & 9 & 3 & 18 & 3 \\
7 & 11 & 1 & 19 & 2.71429 \\
8 & 11 & 1 & 20 & 2.5 \\
9 & 11 & 1 & 21 & 2.33333 \\
10 & 11 & 1 & 22 & 2.2 \\
11 & 11 & 11 & 33 & 3 \\
12 & 21 & 3 & 36 & 3 \\
13 & 23 & 1 & 37 & 2.84615 \\
14 & 23 & 1 & 38 & 2.71429 \\
15 & 23 & 1 & 39 & 2.6 \\
16 & 23 & 1 & 40 & 2.5 \\
17 & 23 & 1 & 41 & 2.41176 \\
18 & 23 & 1 & 42 & 2.33333 \\
19 & 23 & 1 & 43 & 2.26316 \\
20 & 23 & 1 & 44 & 2.2 \\
21 & 23 & 1 & 45 & 2.14286 \\
22 & 23 & 1 & 46 & 2.09091 \\
23 & 23 & 23 & 69 & 3 \\
24 & 45 & 3 & 72 & 3 \\
25 & 47 & 1 & 73 & 2.92 \\
26 & 47 & 1 & 74 & 2.84615 \\
27 & 47 & 1 & 75 & 2.77778 \\
28 & 47 & 1 & 76 & 2.71429 \\
29 & 47 & 1 & 77 & 2.65517 \\
30 & 47 & 1 & 78 & 2.6 \\
31 & 47 & 1 & 79 & 2.54839 \\
32 & 47 & 1 & 80 & 2.5 \\
\end{array}$
&
$\begin{array}{rrcrl}
n & \Delta(n) & g(n) & a(n) & a(n)/n \\ \hline
33 & 47 & 1 & 81 & 2.45455 \\
34 & 47 & 1 & 82 & 2.41176 \\
35 & 47 & 1 & 83 & 2.37143 \\
36 & 47 & 1 & 84 & 2.33333 \\
37 & 47 & 1 & 85 & 2.2973 \\
38 & 47 & 1 & 86 & 2.26316 \\
39 & 47 & 1 & 87 & 2.23077 \\
40 & 47 & 1 & 88 & 2.2 \\
41 & 47 & 1 & 89 & 2.17073 \\
42 & 47 & 1 & 90 & 2.14286 \\
43 & 47 & 1 & 91 & 2.11628 \\
44 & 47 & 1 & 92 & 2.09091 \\
45 & 47 & 1 & 93 & 2.06667 \\
46 & 47 & 1 & 94 & 2.04348 \\
47 & 47 & 47 & 141 & 3 \\
48 & 93 & 3 & 144 & 3 \\
49 & 95 & 1 & 145 & 2.95918 \\
50 & 95 & 5 & 150 & 3 \\
51 & 99 & 3 & 153 & 3 \\
52 & 101 & 1 & 154 & 2.96154 \\
53 & 101 & 1 & 155 & 2.92453 \\
54 & 101 & 1 & 156 & 2.88889 \\
\vdots & \vdots & \vdots & \vdots & \vdots \\
99 & 101 & 1 & 201 & 2.0303 \\
100 & 101 & 1 & 202 & 2.02 \\
101 & 101 & 101 & 303 & 3 \\
102 & 201 & 3 & 306 & 3 \\
103 & 203 & 1 & 307 & 2.98058 \\
104 & 203 & 1 & 308 & 2.96154 \\
105 & 203 & 7 & 315 & 3 \\
106 & 209 & 1 & 316 & 2.98113 \\
\end{array}$
\end{tabular}
\end{center}
\caption{The first few terms for $a(1) = 7$.}\label{table}
\end{table}

\section{Initial observations}\label{initial}

In order to reveal several key features, it is worth recapitulating the experimental process that led to the discovery of the proof that $a(n) - a(n-1)$ is always $1$ or prime.  For brevity, let $g(n) = a(n) - a(n - 1) = \gcd(n, a(n-1))$ so that $a(n) = a(n-1) + g(n)$.  Table~\ref{table} lists the first few values of $a(n)$ and $g(n)$ as well as of the quantities $\Delta(n) = a(n-1) - n$ and $a(n)/n$, whose motivation will become clear presently.  Additional features of Table~\ref{table} not vital to the main result are discussed in \scn{primes}.

One observes from the data that $g(n)$ contains long runs of consecutive $1$s.  On such a run, say if $g(n) = 1$ for $n_1 < n < n_1 + k$, we have
\begin{equation}\label{runofones}
	a(n) = a(n_1) + \sum_{i=1}^{n-n_1} g(n_1 + i) = a(n_1) + (n - n_1),
\end{equation}
so the difference $a(n) - n =  a(n_1) - n_1$ is invariant in this range.  When the next nontrivial $\gcd$ does occur, we see in Table~\ref{table} that it has some relationship to this difference.  Indeed, it appears to divide
\[
	\Delta(n) := a(n-1) - n =  a(n_1) - 1 - n_1.
\]
For example $3 \mid 21$, $23 \mid 23$, $3 \mid 45$, $47 \mid 47$, etc.  This observation is easy to prove and is a first hint of the shortcut mentioned in \scn{introduction}.

Restricting attention to steps where the $\gcd$ is nontrivial, one notices that $a(n) = 3n$ whenever $g(n) \neq 1$.  This fact is the central ingredient in the proof of the lemma, and it suggests that $a(n)/n$ may be worthy of study.  We pursue this in \scn{global}.

Another important observation can be discovered by plotting the values of $n$ for which $g(n) \neq 1$, as in Figure~\ref{clusters}.  They occur in clusters, each cluster initiated by a large prime and followed by small primes interspersed with $1$s.  The ratio between the index $n$ beginning one cluster and the index ending the previous cluster is very nearly $2$, which causes the regular vertical spacing seen when plotted logarithmically.  With further experimentation one discovers the reason for this, namely that when $2n - 1 = p$ is prime for $g(n) \neq 1$, such a ``large gap'' between nontrivial $\gcd$s occurs (demarcating two clusters) and the next nontrivial $\gcd$ is $g(p) = p$.  This suggests looking at the quantity $2n - 1$ (which is $\Delta(n+1)$ when $a(n) = 3n$), and one guesses that in general the next nontrivial $\gcd$ is the smallest prime divisor of $2n-1$.

\begin{figure}
	\includegraphics{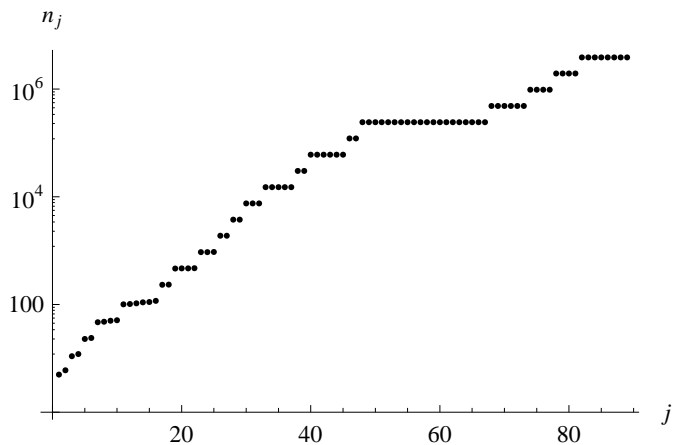}
	\caption{Logarithmic plot of $n_j$, the $j$th value of $n$ for which $a(n) - a(n-1) \neq 1$, for the initial condition $a(1) = 7$.  The regularity of the vertical gaps between clusters indicates local structure in the sequence.}\label{clusters}
\end{figure}

\section{Recurring structure}\label{local}

We now establish the observations of the previous section, treating the recurrence~(\ref{rec}) as a discrete dynamical system on pairs $(n, a(n))$ of integers.  We no longer assume $a(1) = 7$; a general initial condition for the system specifies integer values for $n_1$ and $a(n_1)$.

Accordingly, we may broaden the result:  In the previous section we observed that $a(n)/n=3$ is a significant recurring event; it turns out that $a(n)/n=2$ plays the same role for other initial conditions (for example, $a(3) = 6$).  The following lemma explains the relationship between one occurrence of this event and the next, allowing the elimination of the intervening run of $1$s.  We need only know the smallest prime divisor of $\Delta(n_1 + 1)$.

\begin{lemma}
Let $r \in \{2, 3\}$ and $n_1 \geq \frac{3}{r-1}$.  Let $a(n_1) = r n_1$, and for $n > n_1$ let
\[
	a(n) = a(n-1) + \gcd(n, a(n-1))
\]
and $g(n) = a(n) - a(n - 1)$.
Let $n_2$ be the smallest integer greater than $n_1$ such that $g(n_2) \neq 1$.  Let $p$ be the smallest prime divisor of
\[
	\Delta(n_1 + 1) = a(n_1) - (n_1 + 1) = (r - 1) n_1 - 1.
\]
Then
\begin{itemize}
	\item[(a)] $n_2 = n_1 + \frac{p-1}{r-1}$,
	\item[(b)] $g(n_2) = p$, and
	\item[(c)] $a(n_2) = r n_2$.
\end{itemize}
\end{lemma}

Brief remarks on the condition $(r - 1) n_1 \geq 3$ are in order.  Foremost, this condition guarantees that the prime $p$ exists, since $(r - 1) n_1 - 1 \geq 2$.  However, we can also interpret it as a restriction on the initial condition.  We stipulate $a(n_1) = r n_1 \neq n_1 + 2$ because otherwise $n_2$ does not exist; note however that among positive integers this excludes only the two initial conditions $a(2) = 4$ and $a(1) = 3$.  A third initial condition, $a(1) = 2$, is eliminated by the inequality; most of the conclusion holds in this case (since $n_2 = g(n_2) = a(n_2)/n_2 = 2$), but because $(r - 1) n_1 - 1 = 0$ it is not covered by the following proof.

\begin{proof}
Let $k = n_2 - n_1$.  We show that $k = \tfrac{p-1}{r-1}$.  Clearly $\tfrac{p-1}{r-1}$ is an integer if $r = 2$; if $r = 3$ then $(r - 1) n_1 - 1$ is odd, so $\tfrac{p-1}{r-1}$ is again an integer.

By \eqn{runofones}, for $1 \leq i \leq k$ we have $g(n_1+i) = \gcd(n_1 + i, r n_1 - 1 + i)$.  Therefore, $g(n_1+i)$ divides both $n_1 + i$ and $r n_1 - 1 + i$, so $g(n_1+i)$ also divides both their difference
\[
	(r n_1 - 1 + i) - (n_1 + i) = (r-1) n_1 - 1
\]
and the linear combination
\[
	r \cdot (n_1 + i) - (r n_1 - 1 + i) = (r-1) i + 1.
\]
We use these facts below.

$k \geq \tfrac{p-1}{r-1}$:  Since $g(n_1 + k)$ divides $(r-1) n_1 - 1$ and by assumption $g(n_1 + k) \neq 1$, we have $g(n_1 + k) \geq p$.  Since $g(n_1 + k)$ also divides $(r-1) k + 1$, we have
\[
	p \leq g(n_1 + k) \leq (r-1) k + 1.
\]

$k \leq \tfrac{p-1}{r-1}$:  Now that $g(n_1+i) = 1$ for $1 \leq i < \tfrac{p-1}{r-1}$, we show that $i = \tfrac{p-1}{r-1}$ produces a nontrivial $\gcd$.  We have
\begin{align*}
	g(n_1 + \tfrac{p-1}{r-1}) &= \gcd\!\left( n_1 + \tfrac{p-1}{r-1}, r n_1 - 1 + \tfrac{p-1}{r-1} \right) \\
	&= \gcd\!\left( \frac{((r-1) n_1 - 1) + p}{r-1}, \frac{r \cdot \left((r-1) n_1 - 1\right) + p}{r-1} \right).
\end{align*}
By the definition of $p$, $p \mid ((r-1) n_1 - 1)$ and $p \nmid (r-1)$.  Thus $p$ divides both arguments of the $\gcd$, so $g(n_1 + \tfrac{p-1}{r-1}) \geq p$.

Therefore $k = \tfrac{p-1}{r-1}$, and we have shown (a).  On the other hand, $g(n_1 + \frac{p-1}{r-1})$ divides $(r-1) \cdot \tfrac{p-1}{r-1} + 1 = p$, so in fact $g(n_1 + \tfrac{p-1}{r-1}) = p$, which is (b).  We now have $g(n_2) = p = (r-1) k + 1$, so to obtain (c) we compute
\begin{align*}
	a(n_2) &= a(n_2 - 1) + g(n_2) \\
	&= (r n_1 - 1 + k) + ((r-1) k + 1) \\
	&= r (n_1 + k) \\
	&= r n_2. \qedhere
\end{align*}
\end{proof}

We immediately obtain the following result for $a(1) = 7$; one simply computes $g(2) = g(3) = 1$, and $a(3)/3 = 3$ so the lemma applies inductively thereafter.

\begin{theorem}
Let $a(1) = 7$.  For each $n \geq 2$, $a(n) - a(n-1)$ is $1$ or prime.
\end{theorem}

Similar results can be obtained for many other initial conditions, such as $a(1) = 4$, $a(1) = 8$, etc.  Indeed, most small initial conditions quickly produce a state in which the lemma applies.

\section{Transience}\label{global}

However, the statement of the theorem is false for general initial conditions.  Two examples of non-prime $\gcd$s are $g(18) = 9$ for $a(1) = 532$ and $g(21) = 21$ for $a(1) = 801$.  With additional experimentation one does however come to suspect that $g(n)$ is eventually $1$ or prime for every initial condition.

\begin{conjecture}
If $n_1 \geq 1$ and $a(n_1) \geq 1$, then there exists an $N$ such that $a(n) - a(n-1)$ is $1$ or prime for each $n > N$.
\end{conjecture}

The conjecture asserts that the states for which the lemma of \scn{local} does not apply are transient.  To prove the conjecture, it would suffice to show that if $a(n_1) \neq n_1 + 2$ then $a(N)/N$ is $1$, $2$, or $3$ for some $N$:  If $a(N) = N + 2$ or $a(N)/N = 1$, then $g(n) = 1$ for $n > N$, and if $a(N)/N$ is $2$ or $3$, then the lemma applies inductively.  Thus we should try to understand the long-term behavior of $a(n)/n$.  We give two propositions in this direction.

Empirical data show that when $a(n)/n$ is large, it tends to decrease.  The first proposition states that $a(n)/n$ can never cross over an integer from below.

\begin{proposition}\label{upperbound}
If $n_1 \geq 1$ and $a(n_1) \geq 1$, then $a(n)/n \leq \lceil a(n_1)/n_1 \rceil$ for all $n \geq n_1$.
\end{proposition}

\begin{proof}
Let $r = \lceil a(n_1)/n_1 \rceil$.  We proceed inductively; assume that $a(n-1)/(n-1) \leq r$.  Then
\[
	r n - a(n-1) \geq r \geq 1.
\]
Since $g(n)$ divides the linear combination $r \cdot n - a(n-1)$, we have
\[
	g(n) \leq r n - a(n - 1);
\]
thus
\[
	a(n) = a(n - 1) + g(n) \leq r n. \qedhere
\]
\end{proof}

\begin{figure}
	\includegraphics{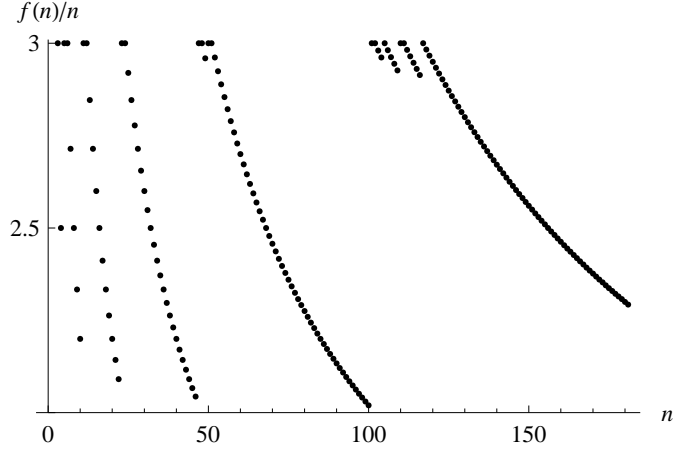}
	\caption{Plot of $a(n)/n$ for $a(1) = 7$.  Proposition~\ref{lowerbound} establishes that $a(n)/n > 2$.}\label{ratio}
\end{figure}

From \eqn{runofones} in \scn{initial} we see that $g(n_1+i) = 1$ for $1 \leq i < k$ implies that $a(n_1+i)/(n_1+i) = (a(n_1)+i)/(n_1+i)$, and so $a(n)/n$ is strictly decreasing in this range if $a(n_1) > n_1$.  Moreover, if the nontrivial $\gcd$s are overall sufficiently few and sufficiently small, then we would expect $a(n)/n \to 1$ as $n$ gets large; indeed the hyperbolic segments in Figure~\ref{ratio} have the line $a(n)/n = 1$ as an asymptote.

However, in practice we rarely see this occurring.  Rather, $a(n_1)/n_1 > 2$ seems to almost always imply that $a(n)/n > 2$ for all $n \geq n_1$.  Why is this the case?

Suppose the sequence of ratios crosses 2 for some $n$:  $a(n)/n > 2 \geq a(n+1)/(n+1)$.  Then
\[
	2 \geq \frac{a(n+1)}{n+1} = \frac{a(n) + \gcd(n+1, a(n))}{n+1} \geq \frac{a(n) + 1}{n+1},
\]
so $a(n) \leq 2n + 1$.  Since $a(n) > 2 n$, we are left with $a(n) = 2n + 1$; and indeed in this case we have
\[
	\frac{a(n+1)}{n+1} = \frac{2n + 1 + \gcd(n+1, 2n+1)}{n+1} = \frac{2n + 2}{n+1} = 2.
\]
The task at hand, then, is to determine whether $a(n) = 2n + 1$ can happen in practice.  That is, if $a(n_1) > 2n_1 + 1$, is there ever an $n > n_1$ such that $a(n) = 2n + 1$?  Working backward, let $a(n) = 2n + 1$.  We will consider possible values for $a(n-1)$.

If $a(n-1) = 2n$, then
\[
	2n + 1 = a(n) = 2n + \gcd(n, 2n) = 3n,
\]
so $n=1$.  The state $a(1) = 3$ is produced after one step by the initial condition $a(0) = 2$ but is a moot case if we restrict to positive initial conditions.

If $a(n-1) < 2n$, then $a(n-1) = 2n - j$ for some $j \geq 1$.  Then
\[
	2n + 1 = a(n) = 2n - j + \gcd(n, 2n - j),
\]
so $j + 1 = \gcd(n, 2n - j)$ divides $2 \cdot n - (2n-j) = j$.  This is a contradiction.

Thus for $n > 1$ the state $a(n) = 2n + 1$ only occurs as an initial condition, and we have proved the following.

\begin{proposition}\label{lowerbound}
If $n_1 \geq 1$ and $a(n_1) > 2n_1 + 1$, then $a(n)/n > 2$ for all $n \geq n_1$.
\end{proposition}

In light of these propositions, the largest obstruction to the conjecture is showing that $a(n)/n$ cannot remain above $3$ indefinitely.  Unfortunately, this is a formidable obstruction:

The only distinguishing feature of the values $r = 2$ and $r = 3$ in the lemma is the guarantee that $\tfrac{p-1}{r-1}$ is an integer, where $p$ is again the smallest prime divisor of $(r-1) n_1 - 1$.  If $r \geq 4$ is an integer and $(r-1) \mid (p-1)$, then the proof goes through, and indeed it is possible to find instances of an integer $r \geq 4$ persisting for some time; in fact a repetition can occur even without the conditions of the lemma.  Searching in the range $1 \leq n_1 \leq 10^4$, $4 \leq r \leq 20$, one finds the example $n_1 = 7727$, $r = 7$, $a(n_1) = r n_1 = 54089$, in which $a(n)/n = 7$ reoccurs eleven times (the last at $n = 7885$).

The evidence suggests that there are arbitrarily long such repetitions of integers $r \geq 4$.  With the additional lack of evidence of global structure that might control the number of these repetitions, it is possible that, when phrased as a parameterized decision problem, the conjecture becomes undecidable.  Perhaps this is not altogether surprising, since the experience with discrete dynamical systems (not least of all the Collatz $3n+1$ problem) is frequently one of presumed inability to significantly shortcut computations.

The next best thing we can do, then, is speed up computation of the transient region so that one may quickly establish the conjecture for specific initial conditions.  It is a pleasant fact that the shortcut of the lemma can be generalized to give the location of the next nontrivial $\gcd$ without restriction on the initial condition, although naturally we lose some of the benefits as well.

In general one can interpret the evolution of \eqn{rec} as repeatedly computing for various $n$ and $a(n-1)$ the minimal $k \geq 1$ such that $\gcd(n + k, a(n-1) + k) \neq 1$, so let us explore this question in isolation.  Let $a(n-1) = n+\Delta$ (with $\Delta \geq 1$); we seek $k$.  (The lemma determines $k$ for the special cases $\Delta = n-1$ and $\Delta = 2n-1$.)

Clearly $\gcd(n + k, n + \Delta + k)$ divides $\Delta$.

Suppose $\Delta = p$ is prime; then we must have $\gcd(n + k, n + p + k) = p$.  This is equivalent to $k \equiv -n \mod p$.  Since $k \geq 1$ is minimal, then $k = \Mod_1(-n,p)$, where $\Mod_j(a,b)$ is the unique number $x \equiv a \mod b$ such that $j \leq x < j + b$.

Now consider a general $\Delta$.  A prime $p$ divides $\gcd(n + i, n + \Delta + i)$ if and only if it divides both $n + i$ and $\Delta$.  Therefore
\[
	\{\, i : \gcd(n + i, n + \Delta + i) \neq 1 \,\} = \bigcup_{p \mid \Delta} \; (-n + p \, \Z).
\]
Calling this set $I$, we have
\[
	k = \min \, \{\, i \in I : i \geq 1 \,\} = \min \, \{\, \Mod_1(-n,p) : p \mid \Delta \,\}.
\]
Therefore (as we record in slightly more generality) $k$ is the minimum of $\Mod_1(-n,p)$ over all primes dividing $\Delta$.

\begin{proposition}\label{gcd}
Let $n \geq 0$, $\Delta \geq 2$, and $j$ be integers.  Let $k \geq j$ be minimal such that $\gcd(n + k, n + \Delta + k) \neq 1$.  Then
\[
	k = \min \, \{\, \Mod_j(-n,p) : \textnormal{$p$ is a prime dividing $\Delta$} \,\}.
\]
\end{proposition}

\section{Primes}\label{primes}

We conclude with several additional observations that can be deduced from the lemma regarding the prime $p$ that occurs as $g(n_2)$ under various conditions.

We return to the large gaps observed in Figure~\ref{clusters}.  A large gap occurs when $(r-1) n_1 - 1 = p$ is prime, since then $n_2 - n_1 = \tfrac{p-1}{r-1}$ is maximal.  In this case we have $n_2 = \tfrac{2p}{r-1}$, so since $n_2$ is an integer and $p > r-1$ we also see that $(r-1) n_1 - 1$ can only be prime if $r$ is $2$ or $3$.  Thus large gaps only occur for $r \in \{2, 3\}$.

Table~\ref{table} suggests two interesting facts about the beginning of each cluster of primes after a large gap:
\begin{itemize}
	\item $p = g(n_2) \equiv 5 \mod 6$.
	\item The next nontrivial $\gcd$ after $p$ is always $g(n_2 + 1) = 3$.
\end{itemize}
The reason is that when $r = 3$, eventually we have $a(n) \equiv n \mod 6$, with exceptions only when $g(n) \equiv 5 \mod 6$ (in which case $a(n) \equiv n + 4 \mod 6$).  In the range $n_1 < n < n_2$ we have $g(n) = 1$, so $p = 2n_1 - 1 = \Delta(n) = a(n-1) - n \equiv 5 \mod 6$ and
\begin{align*}
	g(n_2 + 1) &= \gcd(n_2 + 1, a(n_2)) \\
	&= \gcd(p + 1, 3 p) \\
	&= 3.
\end{align*}
An analogous result holds for $r = 2$ and $n_1 - 1 = p$ prime:  $g(n_2) = p \equiv 5 \mod 6$, $g(n_2 + 1) = 1$, and $g(n_2 + 2) = 3$.

In fact, this analogy suggests a more general similarity between the two cases $r = 2$ and $r = 3$:  An evolution for $r = 2$ can generally be emulated (and actually computed twice as quickly) by $r' = 3$ under the transformation
\begin{align*}
	n' &= n/2, \\
	a'(n') &= a(n) - n/2
\end{align*}
for even $n$ (discarding odd $n$).  One verifies that the conditions and conclusions of the lemma are preserved; in particular
\[
	\frac{a'(n')}{n'} = 2 \cdot \frac{a(n)}{n} - 1.
\]
For example, the evolution from initial condition $a(4) = 8$ is emulated by the evolution from $a'(1) = 7$ for $n = 2n' \geq 6$.

One wonders whether $g(n)$ takes on all primes.  For $r = 3$, clearly the case $p = 2$ never occurs since $2 n_1 - 1$ is odd.  Furthermore, for $r = 2$, the case $p = 2$ can only occur once for a given initial condition:  A simple checking of cases shows that $n_2$ is even, so applying the lemma to $n_2$ we find $n_2 - 1$ is odd (at which point the evolution can be emulated by $r' = 3$).

We conjecture that all other primes occur.  After ten thousand applications of the shortcut starting from the initial condition $a(1) = 7$, the smallest odd prime that has not yet appeared is $587$.

For general initial conditions the results are similar, and one quickly notices that evolutions from different initial conditions frequently converge to the same evolution after some time, reducing the number that must be considered.  For example, $a(1) = 4$ and $a(1) = 7$ converge after two steps to $a(3) = 9$.  One can use the shortcut to feasibly track these evolutions for large values of $n$ and thereby estimate the density of distinct evolutions.  In the range $2^2 \leq a(1) \leq 2^{13}$ one finds that there are only $203$ equivalence classes established below $n = 2^{23}$, and no two of these classes converge below $n = 2^{60}$.  It therefore appears that disjoint evolutions are quite sparse.  Sequence \seqnum{A134162} is the sequence of minimal initial conditions for these equivalence classes.

\section*{Acknowledgement}

Thanks are due to an anonymous referee, whose critical comments greatly improved the exposition of this paper.

\end{document}